\documentclass[a4paper,12pt,reqno,oneside]{amsart} 

\usepackage{setspace} 
\usepackage{typearea} 
\usepackage{amscd,amssymb,verbatim} 

\usepackage{xr-hyper}
\usepackage{cite}
\usepackage[colorlinks,backref,final,bookmarksnumbered,bookmarks]{hyperref}
\usepackage[backrefs]{amsrefs}

\usepackage{ifthen}
\newcommand{\arxiv}[2][]{\ifthenelse{\equal{#1}{}}
{\href{http://arxiv.org/abs/#2}{\tt arXiv:#2}}
{\href{http://arxiv.org/abs/math/#2}{\tt arXiv:math.#1/#2}}}

\theoremstyle{plain}
\newtheorem{theorem}{Theorem}[section]
\newtheorem{lemma}[theorem]{Lemma}

\newtheorem{corollary}[theorem]{Corollary}
\newtheorem{proposition}[theorem]{Proposition}
\newtheorem{problem}[theorem]{Problem}

\newtheorem{maintheorem}{Theorem}
 
\newtheorem*{addendum*}{Addendum}

\theoremstyle{definition}
\newtheorem{example}[theorem]{Example}

\newtheoremstyle{remark}
{}{}{}{}{\itshape}{}{ }{\thmname{#1}\thmnumber{ \itshape #2.}}
\theoremstyle{remark}
\newtheorem{remark}[theorem]{Remark}

\DeclareMathOperator{\id}{id}

\DeclareMathOperator{\Hom}{Hom}

\def\x{\times}
\def\but{\setminus}

\def\eps{\varepsilon}
\def\phi{\varphi}

\def\emptyset{\varnothing}
\renewcommand{\:}{\colon}
\def\Cl#1{\overline{#1}}
\def\xr#1{\xrightarrow{#1}}

\def\R{\mathbb{R}}
\def\Z{\mathbb{Z}}

\def\C{\mathbb{C}}

\begin{document}

\title{Lifting generic maps to embeddings. Triangulation and smoothing}
\author{Sergey A. Melikhov}
\address{Steklov Mathematical Institute of Russian Academy of Sciences,
ul.\ Gubkina 8, Moscow, 119991 Russia}
\email{melikhov@mi-ras.ru}

\begin{abstract}
We show that if a non-degenerate PL map $f\:N\to M$ lifts to a topological embedding in $M\x\R^k$ then it lifts to 
a PL embedding in there.
We also show that if a stable smooth map $N^n\to M^m$, $m\ge n$, lifts to a topological embedding in $M\x\R$, then 
it lifts to a smooth embedding in there.
\end{abstract}

\maketitle

\section{Introduction} \label{sec1}

Let $f\:N\to M$ be a continuous, piecewise linear or smooth map. (By ``smooth'' we will always mean $C^\infty$.)
We say that $f$ is a topological/PL/smooth {\it $k$-prem} ($k$-codimensionally
{\bf pr}ojected {\bf em}bedding) if there exists a map $g\:N\to\R^k$
such that $f\x g\:N\to M\x\R^k$ is a topological/PL/smooth embedding.%
\footnote{That is, a continuous/PL/smooth map which is a homeomorpism/PL homeomorphism/diffeomorphism 
onto its image.}
When the choice of a category is irrelevant, we will speak simply of ``$k$-prems''.
The abbreviation ``prem'' was coined by A. Sz\H ucs in the 90s (see \cite{Akh}, \cite{Sz}), while the notion 
itself is older \cite{BH}, \cite{DH}, \cite{H1}, \cite{H2}, \cite{KW},  \cite{Pe}, \cite{Po}, \cite{Re}, 
\cite{Si}.
Other related work includes \cite{ARS}, \cite{FP}, \cite{Ku}, \cite{Ni}, \cite{RS}, \cite{ST}, \cite{Sa}, 
\cite{Sk}, \cite{TV}, \cite{Ya}.
Some aspects of the theory of $k$-prems are surveyed in the introductions of the recent papers 
\cite{AM}, \cite{M3}.

The main goal of the present paper is to study the difference between topological, PL and smooth $k$-prems.
This paper is a shorter companion of the longer paper \cite{partII}, which addresses the question
of when a given map is actually a $k$-prem.

\begin{maintheorem}\label{th2}
A non-degenerate%
\footnote{A PL map is called {\it non-degenerate} if it has no point-inverses of dimension $>0$.}
PL map between compact polyhedra is a topological $k$-prem if and only if it is a PL $k$-prem.
\end{maintheorem}

Theorem \ref{th2} is contained in Theorem \ref{4.1}.
As a byproduct of its proof we also obtain

\begin{theorem}
Let $f\:N\to M$ be a non-degenerate PL map between compact polyhedra. 
The space of topological embeddings $N\to M\x\R^k$ that lift $f$ is locally contractible.
\end{theorem}

There is also a parallel result for PL embeddings that lift $f$ (see Corollary \ref{loc-contr}).

A smooth map $f\:N\to M$ is called {\it stable} if it has a neighborhood $U$ in $C^\infty(N,M)$ 
such that for every $g\in U$ there exist diffeomorphisms $\phi\:N\to N$ and $\psi\:M\to M$ such that
$\psi f=g\phi$.
The theory of stable smooth maps is exposed in detail in a number of textbooks, including 
\cite{GG} and \cite{AGV}.
We sometimes indicate dimensions of manifolds and polyhedra by the superscript.

\begin{problem} \label{prob} If a stable smooth map $f\:N^n\to M^m$, $m\ge n$, where $N$ is compact,
is a topological $k$-prem, is it a smooth $k$-prem? 
\end{problem}

Without the hypothesis $m\ge n$ the answer would be negative.
(If $N$ embeds in $\R^k$ topologically but not smoothly, then the map $N\to\R^0$ is a topological $k$-prem 
but not a smooth $k$-prem.)

On the other hand, for $k=0$ the answer is affirmative, even if ``stable'' (=$C^\infty$-stable) is weakened to
$C^0$-stable.%
\footnote{Indeed, suppose that $f$ is a $C^0$-stable smooth map which is a topological embedding but not a smooth embedding.
Then $f$ has non-injective differential at some point.
Hence it is $C^\infty$-approximable by smooth maps that are not injective.
So, being $C^0$-stable, it cannot be injective itself, which is a contradiction.}
When $f$ is a smooth immersion, the answer is also affirmative.%
\footnote{Indeed, if $g\:N\to\R^k$ is a map such that $f\x g\:N\to M\x\R^k$ is injective, then for every smooth map 
$g'\:N\to\R^k$ sufficiently $C^0$-close to $g$ the map $f\x g'$ is clearly injective, hence a smooth embedding.}

Theorem \ref{II:th1}(c) of the companion paper \cite{partII} implies an affirmative answer to Problem \ref{prob} in 
the case where $2(m+k)\ge 3(n+1)$ and $3n-2m\le k$.

\begin{maintheorem} \label{th2'} A stable smooth map $f\:N^n\to M^m$, $m\ge n$, is a topological $1$-prem 
if and only if it is a smooth $1$-prem. 
\end{maintheorem}

Theorem \ref{th2'} is contained in Theorem \ref{smoothing}.
Some variation of its proof also yields

\begin{theorem} \label{prim} Let $f\:N^n\to M^m$, $m\ge n$, be a stable smooth map.
If $k\ge 2$, assume additionally that $f$ is a corank one map.%
\footnote{That is, $\dim\ker df_x\le 1$ for each $x\in N$.}
Then $f$ lifts to a topological immersion%
\footnote{That is, a map which embeds some neighborhood of each point of the domain.}
$N\to M\x\R^k$ if and only if it lifts to a smooth immersion $N\to M\x\R^k$.
\end{theorem}

Let us note that, for instance, every stable smooth map $f\:N^4\to\R^5$ is a corank one map 
(see \cite{GG}*{VI.5.2}), so we get, for instance, that if $f$ lifts to a topological immersion 
in $\R^7$, then it also lifts to a smooth immersion in there.

As the author learned from the referee, it is not hard to prove Theorem \ref{prim}
by standard methods of immersion theory (namely, those in Gromov's book \cite{Gr}).
While it may be not so easy to extract an explicit proof from the literature, the referee has kindly provided
a very clear and explicit, and largely self-contained proof (this is the subject of \S\ref{smoothing-referee}).
On the other hand, the author's original proof could be more relevant to Problem \ref{prob}.

\begin{remark}
Going in this direction, if a stable smooth map $f\:N^n\to M^m$, $m\ge n$, lifts to a smooth immersion 
$N\to M\x\R^k$, then it is easy to see that $\dim\ker df_x\le k$ for each $x\in N$, and if $\Sigma^k_f$ denotes 
the set of all $x\in N$ such that $\dim\ker df_x=k$, then $\ker df$ is trivial as a $k$-plane bundle over 
$\Sigma^k_f$.
On the other hand, if $f$ lifts to a topological immersion $N\to M\x\R^k$, it turns out that still
$\dim\ker df_x\le k$ for each $x\in N$, and also $\id\:\Sigma^k_f\to\Sigma^k_f$ is covered by 
a $\Z/2$-equivariant fiberwise map from the spherical bundle over $\Sigma^k_f$ consisting of all unit vectors
in $\ker df$ to the trivial $S^{k-1}$-bundle over $\Sigma^k_f$, where $\Z/2$ acts antipodally on the fibers 
(see the proof of Lemma \ref{morin projection}).
However this map need not be linear on the fibers, in contrast to the smooth case.
This points at one possible approach to Problem \ref{prob}.
\end{remark}

\begin{remark} P. M. Akhmetiev recently announced the following result (see \cite{ACR} and its expected update),
which could be relevant to Problem \ref{prob}: there exists a smooth knot $k\:S^{29}\to S^{31}$ such that 
the composition of $S^{29}\xr{k}S^{31}\subset S^{44}$ is not smoothly slice (i.e.\ does not bound a smooth
embedding $D^{30}\to S^{45}$).
Let us note that $S^{29}$ smoothly unknots in $S^{46}$ by Haefliger's theorem  (see \cite{Ad}*{\S VII.4})
and PL unknots in $S^{32}$ by Zeeman's theorem \cite{Ze0}.

The knot $k$ is a Brieskorn sphere.
In more detail, let $V$ be the complex hypersurface in $\C^{16}$ given by the equation $f(z)=0$, where
$f(z_1,\dots,z_{16})=z_1^3+z_2^2+\dots+z_{16}^2$.
Let $\Sigma$ be the intersection of $V$ with a small sphere $S^{31}$ about $0$ given by the equation
$|z_1|^2+\dots+|z_{16}|^2=\eps$.
Then $\Sigma$ is homeomorphic to $S^{29}$ (see \cite{Mi}*{8.5 and 9.1}), and $S^{31}\but\Sigma$ is
not homotopy equivalent to $S^1$ (see \cite{Mi}*{proof of 7.3}).
A Seifert surface $M$ of $\Sigma$ can be described as $\phi^{-1}(pt)$, where $\phi\:S^{31}\but\Sigma\to S^1$ 
is defined by $\phi(z)=f(z)/|f(z)|$ (see \cite{Mi}*{6.1}).%
\footnote{In fact, $\phi$ is a smooth bundle (see \cite{Mi}*{4.8}).
It is also known that $M$ is diffeomorphic to the intersection of $f^{-1}(c)$, where $|c|$ is small, with 
a small ball about $0$ (see \cite{Mi}*{5.11}) and is homotopy equivalent to the join of $S^{14}$ with 
the $3$-point set (see \cite{Mi}*{proof of 9.1}).}
The Kervaire invariant of $M$ is nonzero \cite{Le}*{\S3}, \cite{Mi}*{8.7}.
However, there also exists a closed framed manifold $N^{30}$ with nonzero Kervaire invariant \cite{Br} (see also
\cite{Akh2}).
Hence $\Sigma$ bounds a framed manifold, namely $M\# N$, with zero Kervaire invariant.
Then $\Sigma$ bounds a contractible manifold \cite{KM}*{5.5, 8.4} and hence is $h$-cobordant to $S^{29}$ 
\cite{KM}*{2.3}.
Therefore by Smale's theorem $\Sigma$ is diffeomorphic to $S^{29}$.
The knot $k$ is the composition of this diffeomorphism and the inclusion $\Sigma\subset S^{31}$.
\end{remark}

\section{Triangulation of lifts}\label{sec4}

Let $K$ be a simplicial complex and $K'$ a derived (i.e.\ weighted barycentric) subdivision of $K$.
For a simplex $\sigma$ of $K$, let $\hat\sigma$, or in more detail $\hat\sigma_{K'}$, denote its weighted
barycenter in $K'$.
The {\it dual cone} $\sigma^*$, or in more detail $\sigma^*_{K'}$, is the subcomplex of $K'$ consisting of all 
simplexes of the form $\hat\tau_1*\dots*\hat\tau_n$, where $\sigma\subset\tau_1\subsetneqq\dots\subsetneqq\tau_n$
are simplexes of $K$.
Thus $\sigma^*=\hat\sigma*\partial\sigma^*$, where the {\it derived link} $\partial\sigma^*$ is the subcomplex of 
$K'$ consisting of all simplexes of the form $\hat\tau_1*\dots*\hat\tau_n$, where 
$\sigma\subsetneqq\tau_1\subsetneqq\dots\subsetneqq\tau_n$ are simplexes of $K$.
If $K$ is a combinatorial $n$-manifold and $\sigma$ is a $k$-simplex, then $|\sigma^*|$ is an $(n-k)$-cell 
intersecting $\sigma$ at $\hat\sigma$.

\begin{lemma}\label{triangulation}
Let $f\:P\to Q$ be a non-degenerate simplicial map between finite simplicial complexes and $g\:|P|\to\R^k$ 
be a continuous map such that $f\x g\:|P|\to |Q|\x\R^k$ is an embedding.

Then there exist subdivisions $K$, $L$ of $P$, $Q$ and their derived subdivisions $K'$, $L'$ such that 
$f\:K\to L$ and $f\:K'\to L'$ are simplicial and for any distinct vertices $u$, $v$ of $K$ satisfying 
$f(u)=f(v)$, the convex hulls of $g(|u^*_{K'}|)$ and $g(|v^*_{K'}|)$ are disjoint.
\end{lemma}

This lemma will also be used in the proof of Theorem \ref{th2'}.

\begin{proof}
By the hypothesis, $f$ embeds every simplex of $P$.
Let $P^{(i)}$ and $Q^{(i)}$ denote the unions of all $i$-simplexes of $P$ and of $Q$.
We fix some metric on $|P|$.

Next we shall define subdivisions $K=K_n\triangleright\dots\triangleright K_0\triangleright P$
and $L=L_n\triangleright\dots\triangleright L_0\triangleright Q$, where $n=\dim|P|$, such that $f$ 
is simplicial as a map $K_i\to L_i$ for each $i$.

Let $d_0$ be the maximum of the distance $||g(u)-g(v)||$ over all pairs $(u,v)$ of distinct vertices of $P$ 
such that $f(u)=f(v)$.
Since $g$ is uniformly continuous, there exists an $r_0>0$ such that for any $x,y\in|P|$ at a distance $\le r_0$, 
$||g(x)-g(y)||<d_0/2$.
Let $K_0$ and $L_0$ be any subdivisions of $P$ and $Q$ such that $f\:K_0\to L_0$ is simplicial and every simplex 
of $K_0$ has diameter $<r_0$.
(Here $K_0$ is uniquely determined by $L_0$, and $L_0$ is chosen depending on $r_0$.) 

Let us assume that $K_i$ and $L_i$ are subdivisions of $K_0$ and $L_0$ such that $f\:K_i\to L_i$ is simplicial.
Let $X_i$ be the union of all simplexes of $K_i$ that are disjoint from $P^{(i)}$, and let $U_i$ be 
a neighborhood of $X_i$ whose complement is a neighborhood of $P^{(i)}$.
Let $d_{i+1}$ be the supremum of the distance $||g(x)-g(y)||$ over all pairs $(x,y)$ of distinct points of 
$P^{(i+1)}\cap U_i$ such that $f(x)=f(y)$.
Since $g$ is uniformly continuous, there exists an $r_{i+1}>0$ such that for any $x,y\in|P|$ at a distance 
$\le r_{i+1}$, $||g(x)-g(y)||<d_{i+1}/2$.
Let $K_{i+1}$ and $L_{i+1}$ be any subdivisions of $K_i$ and $L_i$ such that $f\:K_{i+1}\to L_{i+1}$ is simplicial,
$K_{i+1}$ has new vertices only in $X_i$, and every simplex of $K_{i+1}$ contained in $X_i$ 
($\Leftrightarrow$ disjoint from $P^{(i)}$) has diameter $<r_{i+1}$.
Let $V_i$ be a neighborhood of $X_i$ in $U_i$ such that for each simplex $\sigma*\tau$ of
$K_{i+1}$, where $\sigma\subset P^{(i)}$ and $\tau\subset X_i$, the diameter of $V_i\cap(\sigma*\tau)$
is $\le r_{i+1}$.

Let $K=K_n$ and $L=L_n$, where $n=\dim |P|$.
Let $K'$ and $L'$ be any derived subdivisions of $K$ and $L$ such that $f\:K'\to L'$ is simplicial and for each $i$,
every simplex of $K'$ that is disjoint from $P^{(i)}$ lies in $V_i$.
(Let us note that these conditions for different values of $i$ do not follow from each other.)
Then every simplex of $K'$ that is disjoint from $P^{(i)}$ has diameter $\le r_{i+1}$.

Let $v$ be a vertex of $K$ which lies in $P^{(i+1)}\but P^{(i)}$ for some $i$.
Then $|v^*_{K'}|$ lies in the $r_{i+1}$-ball centered at $v$.
Hence $g(|v^*_{K'}|)$ lies in the ball $B_v$ of radius $d_{i+1}/2$ centered at $g(v)$.
If $u$ is a vertex of $K$ such that $f(u)=f(v)$, then $u\in P^{(i+1)}\but P^{(i)}$ since $f$ is non-degenerate.
Then both $u$ and $v$ lie in $P^{(i+1)}\cap X_i$, and hence $B_u\cap B_v=\emptyset$.
Thus $g(|u^*_{K'}|)$ and $g(|v^*_{K'}|)$ have disjoint convex hulls.
\end{proof}

\begin{example}
It would be more convenient if the subdivisions $K'$, $L'$ in Lemma \ref{triangulation} could be chosen to be 
barycentric, but this is not possible in general.
For example, let $f\:[-1,1]\to [0,1]$ be defined by $f(x)=|x|$ and $g\:[-1,1]\to\R$ be defined by 
$g(0)=0$, $g(x)=x(-1+\cos\frac{2\pi}x)$ for $x>0$ and $g(x)=x(1+\cos\frac{2\pi}x)$ for $x<0$.
Then $f\x g\:[-1,1]\to[0,1]\x\R$ is an embedding.
Let us note that for $x>0$ we have $g(x)\le 0$, with $g(x)=0$ precisely when 
$|x|\in\{1,\frac12,\frac13,\dots\}$.
On the other hand, for $x<0$ we have $g(x)\ge 0$, with $g(x)=0$ precisely when 
$|x|\in\{\frac1{1.5},\frac1{2.5},\frac1{3.5},\dots\}$.
If $L$ (which triangulates $Q=[0,1]$) has an edge $e$ with vertex $0$ and other vertex $\eps$, then its 
barycenter is at $\frac\eps2$.
Let us note that for each $\eps\in(0,1]$, the interval $[\frac\eps2,\eps]$ contains a pair of 
consecutive members of the sequence $1,\frac1{1.5},\frac12,\frac1{2.5},\frac13,\frac1{3.5},\dots\dots$.
Consequently, the convex hulls of $g([\frac\eps2,\eps])$ and $g([-\eps,-\frac\eps2])$ are not disjoint.
\end{example}

\begin{theorem}\label{4.1}
Let $f\:N\to M$ be a non-degenerate PL map between compact polyhedra.
Then $f$ is a topological $k$-prem if and only if it is a PL $k$-prem.

Moreover, if $e\:N\to M\x\R^k$ is a topological embedding which lifts $f$, then $e$ is isotopic through 
lifts of $f$ to a PL embedding.

Furthermore, if $e$ is PL on a subpolyhedron $N_0$ of $N$, then the isotopy may be assumed to keep $N_0$ fixed.
\end{theorem}

\begin{proof}
We have $e=f\x g$, where $g\:N\to\R^k$ is the composition of $e$ with the projection.
Let $P$ and $Q$ be triangulations of $N$ and $M$ such that $f\:P\to Q$ is simplicial, $N_0$ is triangulated by 
a subcomplex $P_0$ of $P$ and $g$ is linear on the simplexes of $P_0$.
Let $K$, $L$ and $K'$, $L'$ be the subdivisions given by Lemma \ref{triangulation}.
Let $g_i\:N\to\R^k$ be the map that equals $g$ on the dual cone $|\sigma^*_{K'}|$ of each simplex $\sigma$ of $K$
of dimension $\ge i$, and is extended conically to all $|\sigma^*_{K'}|$ such that $\dim\sigma<i$, in the sense 
that $g_i\big(tx+(1-t)\hat\sigma\big)=tg_i(x)+(1-t)g(\hat\sigma)$ for each $x\in|\partial\sigma^*_{K'}|$.
Then $g_0=g$ and $g_n$ is simplicial on $K'$ (and in particular PL), where $n=\dim N$.
Each $g_i$ is homotopic to $g_{i+1}$ by a version of the Alexander trick.

In fact, these $n$ Alexander tricks can be done independently of each other.
This results in an $n$-homotopy $h_t\:N\to\R^k$, $t\in I^n$, which is defined as follows.
Let us write $t=(t_1,\dots,t_n)$, where each $t_i\in [0,1]$.
Every simplex of $K'$ lies in a simplex $\sigma$ of $K'$ of the form $\sigma=\hat\sigma_0*\dots*\hat\sigma_k$,
where $\sigma_0\subset\dots\subset\sigma_k$ is a full flag of simplexes of $K$ (in particular, 
each $\dim\sigma_i=i$).
Given an $s=(s_1,\dots,s_k)$, where each $s_i\in[0,1]$, let us define $x_i(s)$ recursively by 
$x_k(s)=\hat\sigma_k$ and $x_{i-1}(s)=(1-s_i)\hat\sigma_{i-1}+s_ix_i(s)$, and let us write $x(s)=x_0(s)$.
Let $v_i(s)=x(s_1,\dots,s_i,0,\dots,0)$, and let $w_i(s)$ be the image of $v_i(s)$ under the affine map sending 
each $\hat\sigma_i$ to $g(\hat\sigma_i)$.
(In other words, $w_i(s)=g_n\big(v_i(s)\big)$.)
Let us write $s_i'=\max(s_i,t_i)$, $\bar s_i=s_i/s'_i$ and $\bar s=(\bar s_1,\dots,\bar s_n)$.
Let us note that if each $s_i\le t_i$, then $x(\bar s)$ is the image of $x(s)$ under the affine map 
$v_0(t)*\dots*v_k(t)\to\sigma$ sending each $v_i(t)$ to $\hat\sigma_i$.
Let us define $y_i(s)$ recursively by $y_k(s)=g\big(x(\bar s)\big)$ and 
$y_{i-1}(s)=(1-s'_i)w_{i-1}(\bar s)+s'_iy_i(s)$.
(Thus $w_k(s)$ is not used.)
Then $h_t$ is defined by $h_t\big(x(s)\big)=y_0(s)$.
It is easy to see that $h_t|_{N_0}=g|_{N_0}$, $h_{(1,\dots,1)}=g$, $h_{(0,\dots,0)}=g_n$ and more generally 
each $h_{({\tiny\underbrace{0,\dots,0}_i},1,\dots,1)}=g_i$.

Let us fix some $t\in I^n$.
It is clear from the definition of $h_t$ that $h_t(\sigma)$ lies in the convex hull of $g(\sigma)$ for each
simplex $\sigma$ of $K'$.
Then it follows from Lemma \ref{triangulation} that for any distinct vertices $u$, $v$ of $K$ satisfying 
$f(u)=f(v)$ we have $h_t(|u^*_{K'}|)\cap h_t(|v^*_{K'}|)=\emptyset$.
Let us show that $f\x h_t$ is injective.
Suppose that $h_t(x)=h_t(y)$ for some distinct $x,y\in N$ such that $f(x)=f(y)$.
Let $\sigma$, $\tau$ be the minimal simplexes of $K'$ containing $x$ and $y$.
Then $\sigma=(\sigma\cap\tau)*\tilde\sigma$ and $\tau=(\sigma\cap\tau)*\tilde\tau$ for some 
simplexes $\tilde\sigma$ and $\tilde\tau$ such that $\tilde\sigma\cap\tilde\tau=\emptyset$.
If $\sigma\cap\tau\ne\emptyset$, then there exist unique points $z\in\sigma\cap\tau$ and
$\tilde x\in\tilde\sigma$, $\tilde y\in\tilde\tau$ such that $x\in z*\tilde x$ and 
$y\in z*\tilde y$.%
\footnote{Indeed, since $f$ is non-degenerate, it restricts to affine isomorphisms of $\sigma$
and $\tau$ with $\rho:=f(\sigma)=f(\tau)$.
We have $\rho=\alpha*\beta$, where $\alpha=f(\sigma\cap\tau)$ and 
$\beta=f(\tilde\sigma)=f(\tilde\tau)$.
Now $f(x)\in a*b$ for unique points $a\in\alpha$ and $b\in\beta$.
Let $z$ be the unique preimage of $a$ in $\sigma\cap\tau$ and $\tilde x$, $\tilde y$ be 
the unique preimages of $b$ in $\tilde\sigma$ and $\tilde\tau$.}
Clearly $f(\tilde x)=f(\tilde y)$ and it is not hard to see that $h_t(\tilde x)=h_t(\tilde y)$.
So by considering $\tilde x$ and $\tilde y$ instead of $x$ and $y$ we may assume that $\sigma\cap\tau=\emptyset$.

Then it is easy to see that $\sigma$ and $\tau$ are contained respectively in $|u^*_{K'}|$ and $|v^*_{K'}|$
for some distinct vertices $u$ and $v$ of $K$ such that $f(u)=f(v)$.
(Indeed, we have $\sigma=\hat\sigma_1*\dots*\hat\sigma_k$ and $\tau=\hat\tau_1*\dots*\hat\tau_k$ for some 
simplexes $\sigma_1\subsetneqq\dots\subsetneqq\sigma_n$ and $\tau_1\subsetneqq\dots\subsetneqq\tau_n$ of $K$.
Then $\sigma_1\ne\tau_1$ and $f(\sigma_1)=f(\tau_1)$, so $f(u)=f(v)$ for some vertex $u$ of $\sigma_1$ and
some vertex $v$ of $\tau_1$ such that $u\ne v$.)
Thus $x\in|u^*_{K'}|$ and $y\in|v^*_{K'}|$, where $h_t(|u^*_{K'}|)\cap h_t(|v^*_{K'}|)=\emptyset$, contradicting 
our hypothesis $h_t(x)=h_t(y)$.
\end{proof}

As a byproduct of the proof of Theorem \ref{4.1}, we also obtain the following result (whose proof additionally 
depends on a half-page argument from the companion paper \cite{partII}):

\begin{theorem} Let $f\:N\to M$ be a non-degenerate PL map between compact polyhedra and $e\:N\to M\x\R^k$ 
be a topological embedding which lifts $f$.

Then for each $\eps>0$ there exist a $\delta>0$ and a PL embedding $e^\star\:N\to M\x\R^k$ such that if $X$ is 
a space and $E\:N\x X\to M\x\R^k\x X$ is a topological embedding which lifts $f\x\id_X$ and is $\delta$-close 
to $e\x\id_X$, then $E$ is $\eps$-isotopic to $e^\star\x\id_X$ through lifts of $f\x\id_X$.

Moreover, if $e$ is PL on a subpolyhedron $N_0$ of $N$ and $E|_{N_0\x Y}=e|_{N_0}\x\id_Y$ for some $Y\subset X$, 
then the isotopy may be assumed to keep $N_0\x Y$ fixed.

Furthermore, if $X$ is a polyhedron and $E$ is PL, then the isotopy may be assumed to be PL.
\end{theorem}

\begin{proof} We have $e=f\x g$, where $g\:N\to\R^k$ is the composition of $e$ with the projection.
Let $P$ and $Q$ be triangulations of $N$ and $M$ such that $f\:P\to Q$ is simplicial, the $g$-image of 
the star of every vertex of $P$ is of diameter $\le\eps/2$, and also $N_0$ is triangulated by a subcomplex $P_0$ 
of $P$ and $g$ is linear on the simplexes of $P_0$.
Let $K$, $L$ and $K'$, $L'$ be the subdivisions given by Lemma \ref{triangulation}.
Then there exists a $\delta_1>0$ such that for any distinct vertices $u$, $v$ of $K$ satisfying $f(u)=f(v)$, 
the convex hulls of $g(|u^*_{K'}|)$ and $g(|v^*_{K'}|)$ are at a distance $\ge 3\delta_1$.
Let $g^\star\:N\to\R^k$ equal $g$ on the vertices of $K'$ and be linear on the simplexes of $K'$.
Then $g^\star|_{N_0}=g|_{N_0}$ and by the proof of Theorem \ref{4.1} $e^\star:=f\x g^\star$ is an embedding.
By \cite{partII}*{proof of Proposition \ref{II:emb-stable}} there exists a $\delta_2>0$ such that $f\x\phi$ 
is an embedding for every map $\phi\:N\to\R^k$ that is $2\delta_2$-close to $g^\star$ and linear on 
the simplexes of $K'$.
Let $\delta=\min(\delta_1,\delta_2,\eps/2)$.

Now let $E$ be given by the hypothesis and let $e_x\:N\to M\x\R^k$ be defined by $\big(e_x(p),x\big)=E(p,x)$
for each $x\in X$.
We have $e_x=f\x g_x$, where $g_x\:N\to\R^k$ is the composition of $e_x$ with the projection.
Each $g_x$ is $\delta$-close $g$, so for each vertex $v$ of $K$ the convex hull of $g_x(|v^*_{K'}|)$ lies in 
the $\delta$-neighborhood of the convex hull of $g(|v^*_{K'}|)$.
In particular, for any distinct vertices $u$, $v$ of $K$ satisfying $f(u)=f(v)$, the convex hulls of 
$g_x(|u^*_{K'}|)$ and $g_x(|v^*_{K'}|)$ are disjoint (due to our choice of $\delta_1$).
Let $g_x^\star\:N\to\R^k$ equal $g_x$ on the vertices of $K'$ and be linear on the simplexes of $K'$,
and let us define $E^\star\:N\x X\to M\x\R^k\x X$ by $E^\star(p,x)=\big(f(p),g^\star_x(p),x\big)$.
Then $g_x^\star|_{N_0}=g_x|_{N_0}=g|_{N_0}=g^\star|_{N_0}$ for each $x\in Y$, and by the proof of 
Theorem \ref{4.1} $e_x$ is isotopic through lifts of $f$ to $f\x g_x^\star$ keeping $N_0$ fixed.
Moreover, the resulting isotopy $E_t$ between $E$ and $E^\star$ is continuous, and if $X$ is a polyhedron 
and $E$ is PL, then $E_t$ is a PL isotopy.
If $g_{xt}$ is the linear homotopy between $g_x^\star$ and $g^\star$, then $f\x g_{xt}$ is an isotopy 
(due to our choice of $\delta_2$).
Obviously, the resulting isotopy $E'_t$ between $E^\star$ and $e^\star\x\id_X$ is continuous, and if $X$ 
is a polyhedron and $E$ is PL, then $E'_t$ is a PL isotopy.
For each vertex $v$ of $K$, each time instance of the stacked homotopy 
$g_x\rightsquigarrow g_x^\star\rightsquigarrow g^\star$ sends $|v^*_{K'}|$ into the $\delta$-neighborhood of 
the convex hull of $g(|v^*_{K'}|)$ (by the proof of Theorem \ref{4.1}) and hence is $\eps$-close to $g$ 
(due to our choice of $\eps$).
\end{proof}

\begin{corollary} \label{loc-contr} Let $f\:N\to M$ be a non-degenerate PL map between compact polyhedra.

(a) The space of topological embeddings $N\to M\x\R^k$ which lift $f$ is locally contractible.

(b) Given a topological embedding $e\:N\to M\x\R^k$ which lifts $f$, for each $\eps>0$ there exists
a $\delta>0$ such that for each $n=0,1,\dots$, every PL embedding $N\x\partial B^n\to M\x\R^k\x\partial B^n$ 
which lifts $f\x\id_{\partial B^n}$ and is $\delta$-close to $e\x\id_{\partial B^n}$ bounds a PL embedding 
$N\x B^n\to M\x\R^k\x B^n$ which lifts $f\x\id_{B^n}$ and is $\eps$-close to $e\x\id_{B^n}$.
\end{corollary}

\section{Smoothing of lifts} \label{smoothing-sec}

Given a space $N$, let $\Delta_N=\{(x,x)\in N\x N\}$ and $\tilde N=N\x N\but\Delta_N$.
Given a map $f\:N\to M$, let $\Delta_f=\{(x,y)\in\tilde N\mid f(x)=f(y)\}$, and if $f$ is smooth, 
let $\Sigma_f=\{x\in N\mid\ker df_x\ne 0\}$.
Thus $\Delta_{\Sigma_f}=\{(x,x)\in N\x N\mid\ker df_x\ne 0\}$.

A necessary condition for $f\:N\to M$ to be a $k$-prem is the existence of an equivariant map 
$\tilde g\:\Delta_f\to S^{k-1}$ with respect to the factor exchanging involution on $\Delta_f\subset\tilde N$
and the antipodal involution on $S^{k-1}$.
Namely, $\tilde g(x,y)=\frac{g(y)-g(x)}{||g(y)-g(x)||}$, where $g\:N\to\R^k$ is a map such that
$f\x g\:N\to M\x\R^k$ is an embedding.

\begin{lemma} \label{morin projection}
Let $f\:N^n\to M^m$, $n\le m$, be a stable smooth map that lifts to a topological immersion $g\:N\to M\x\R$.
Then $f$ is a corank one map and $\ker df$ is trivial as a line bundle over $\Sigma_f$.
\end{lemma}

\begin{proof} Let $\check N$ be the usual compactification of $\tilde N=N\x N\but\Delta_N$ by a copy of 
the total space $SN$ of the spherical tangent bundle of $N$ (also known as the Axelrod--Singer compactification;
see \cite{Ro}, \cite{FM}, \cite{partII}*{\S\ref{II:FM-AS}}).
Let $\check\Sigma_f\subset SN$ be the set of all unit tangent vectors of $N$ that lie in $\ker df$.
Since $f$ is stable, by a transversality theorem of F. Ronga \cite{Ro} (see also 
\cite{partII}*{Theorem \ref{II:xx-0} and Corollary \ref{II:identification}}) the closure $\check\Delta_f$ 
of $\Delta_f$ in $\check N$ is a manifold with boundary $\check\Sigma_f$.
The immersion $g$ yields an equivariant map $\tilde g\:\Delta_f\cap U\to S^0$ for some $\Z/2$-invariant 
open neighborhood $U$ of $\Delta_N$ in $N\x N$.
Let $\check U$ be the preimage of $U$ in $\check N$.
Then the manifold with boundary $\check\Delta_f\cap\check U$ is equivariantly homotopy equivalent to its interior
$\Delta_f\cap U$.
Hence $\check\Delta_f\cap\check U$ also admits an equivariant map to $S^0$, and so does its boundary 
$\check\Sigma_f$.

Now suppose that $df_x\:T_xN\to T_{f(x)}M$ has kernel of dimension $\ge 2$ for some $x\in N$.
Then some unit vector $v\in\ker df_x$ can be deformed into $-v$ through unit vectors in $\ker df_x$.
Hence the set $\check\Sigma_f$ of all unit vectors in $\ker df$ admits no equivariant map to $S^0$,
which is a contradiction. 
Thus $f$ is a corank one map.

Finally, suppose that the line bundle $\ker df$ is nontrivial.
Then it is nontrivial over some loop $l$ in $\Sigma_f$.
Then for any point $x$ in $l$, each unit vector $v\in\ker df_x$ deforms into $-v$ upon traversing along $l$.
Hence $\check\Sigma_f$ admits no equivariant map to $S^0$, again.
\end{proof}

\begin{lemma} \label{smoothcurve}
If $f\:N\to M$ is a stable smooth map, then for each $x\in N$ and each $v\in\ker df_x\but 0$ there exists 
a smooth curve $\gamma\:\R\to N$ such that $\gamma(0)=x$, $\gamma'(0)=v$ and 
$f\big(\gamma(t)\big)=f\big(\gamma(-t)\big)$ for each $t\in\R$.
\end{lemma}

This is an easy consequence of a transversality theorem of F. Ronga \cite{Ro} (for the details 
see \cite{partII}*{Theorem \ref{II:xx-0} and Lemma \ref{II:smoothcurve}}).

\begin{lemma} \label{homotopy} 
(a) Let $f\:N\to M$ be a map between topological spaces such that $f^{-1}(p)$ is discrete for each $p\in M$.
Then embedded lifts $f\x g,\,f\x g'\:N\to M\x\R$ of $f$ are isotopic through lifts of $f$ if and only if the maps
$\tilde g,\tilde g'\:\Delta_f\to S^0$ coincide.

(b) Let $f\:N^n\to M^m$, $n\le m$, be a stable smooth map.
Then smoothly embedded lifts $f\x g,\,f\x g'\:N\to M\x\R$ of $f$ are smoothly isotopic through lifts of $f$ 
if and only if the maps $\tilde g,\tilde g'\:\Delta_f\to S^0$ coincide.
\end{lemma}

\begin{proof}[Proof. (a)] Let $g_t=(1-t)g+tg'$.
For each pair $(x,y)\in\Delta_f$ the vectors $g(x)-g(y)$ and $g'(x)-g'(y)$ are of the same sign.
Hence $g_t(x)-g_t(y)$ is also of the same sign, and in particular nonzero, for each $t\in I$.
Thus each $f\x g_t\:N\to M\x\R$ is an embedding.
\end{proof}

\begin{proof}[(b)] Since $g$ and $g'$ are smooth, so is the homotopy $g_t$ constructed in (a).
Since $f\x g_t$ is an isotopy by (a), it suffices to show that each $f\x g_t$ is a smooth immersion.
Since $df_x\x dg_x\:T_xN\to T_{f(x)}N\x\R$ is injective for each $x\in N$, $dg_x|_{\ker df_x}\:\ker df_x\to\R$ 
is an isomorphism for each $x\in\Sigma_f$.
By Lemma \ref{morin projection} $\ker df$ is a trivial line bundle over $\Sigma_f$, so we may identify
each $\ker df_x$ with $\R$.
Let $\hat g\:\Sigma_f\to\R$ be defined by $\hat g(x)=dg_x|_{\ker df_x}(1)$.
Since $f$ is stable, by Lemma \ref{smoothcurve} the sign of $\hat g(x)$ is determined by $\tilde g(y,y')$ for 
a pair $(y,y')\in\Delta_f$ that is sufficiently close to $(x,x)$.
Hence it is the same as that of $\hat g'(x)$.
Then $\hat g_t(x)$ is also of the same sign, and in particular nonzero, for each $t\in I$. 
Thus each $f\x g_t\:N\to M\x\R$ is a smooth immersion.
\end{proof}

Let $f\:N^n\to M^m$, $n\le m$, be a corank one stable smooth map.
Then we have $\Sigma_f=\Sigma^{1,0}_f\cup\Sigma^{1,1,0}_f\cup\dots$, where $\Sigma_f^{1_r,0}$ can be nonempty 
only when $(m-n+1)r\le n$.
By Morin's theorem \cite{Mo} $f$ is locally $C^\infty$-left-right-equivalent at each $p\in\Sigma_f^{1_r,0}$ 
to the map $F_r\:\R^n\to\R^m$, defined by $F_r(t_1,\dots,t_{n-1},x)=(t_1,\dots,t_{n-1},y,z_1,\dots,z_{m-n})$, 
where $y=t_1x+\dots+t_{r-1}x^{r-1}+x^{r+1}$ and each $z_i=t_{ir}x+\dots+t_{ir+r-1}x^r$.

The Morin map $F_r\:\R^n\to\R^m$ has two obvious lifts $\Phi_r^\pm\:\R^n\to\R^{m+1}$, defined by 
$\Phi_r^\pm(t_1,\dots,t_{n-1},x)=(t_1,\dots,t_{n-1},y,z_1,\dots,z_{m-n},\pm x)$, which are smooth embeddings.

So far we were assuming that $r>0$.
But we may also consider $F_0$, which is the inclusion of $\R^n$ onto $\R^n\x 0\subset\R^m$.
Let $\Phi_0^\pm=F_0$.

\begin{theorem} \label{smoothing} Let $f\:N^n\to M^m$, $n\le m$, be a stable smooth map from a compact smooth 
manifold to a smooth manifold. 
Then $f$ is a topological $1$-prem if and only if it is a smooth $1$-prem. 

Moreover, if $e\:N\to M\x\R$ is a topological embedding which lifts $f$, then $e$ is isotopic through 
lifts of $f$ to a smooth embedding.
\end{theorem}

\begin{proof}
We have $e=f\x g$, where $g\:N\to\R$ is the composition of $e$ with the projection.
Since $f$ is stable, by Lemma \ref{smoothcurve} for each $x\in\Sigma_f$ and each $v\in\ker df_x\but\{0\}$ there exists 
a smooth curve $\gamma\:\R\to N$ such that $\gamma(0)=x$, $\gamma'(0)=v$ and $\big(\gamma(t),\gamma(-t)\big)\in\Delta_f$ 
for each $t\in\R\but\{0\}$.
The proof of Lemma \ref{morin projection} yields a trivialization $\eps$ of $\ker df$ as a line bundle over 
$\Sigma_f$ such that $\eps_x(v)$ and $\tilde g\big(\gamma(t),\gamma(-t)\big)$ are of the same sign for all
$x$, $v$ and $\gamma$ as above.

By a theorem of A. Verona \cite{Ve} (see also \cite{partII}*{Theorem \ref{II:Shiota}}) 
we may assume that $f$ is simplicial in some smooth triangulations of $N$ and $M$.
Let $K$, $L$ and $K'$, $L'$ be the subdivisions given by Lemma \ref{triangulation}.
Clearly, each $\Sigma^{1_r,0}_f$ lies in the $(n-r)$-skeleton of $K$.
If $\tau$ is an $i$-simplex of $L$, we may assume (by doing some smoothing) that its dual cone $\tau^*$ with 
respect to $L'$ is a smooth $(m-i)$-disk transverse to $f$.%
\footnote{This assumption helps to simplify notation, but one can do without it by considering appropriate 
open neighborhoods of the dual cones.}
If $\sigma$ is an $i$-simplex of $K$ such that $f(\sigma)=\tau$, then its dual cone $\sigma^*$ with respect to 
$K'$ is a smooth $(n-i)$-disk such that $f(\sigma^*)\subset\tau^*$.
Let $I_\sigma$ denote the convex hull of $g(\sigma^*)$ in $\R$.
If $\sigma'$ is another $i$-simplex of $K$ such that $f(\sigma')=\tau$, then 
$I_\sigma\cap I_{\sigma'}=\emptyset$ by Lemma \ref{triangulation}.
The union of all $\sigma^*$ where $\sigma$ is an $i$-simplex of $K$ such that $f(\sigma)=\tau$ coincides 
with $f^{-1}(\tau^*)$ and will be denoted $\tau^*_f$.

Let $N_i$ be the union of all $\tau^*_f$ where $\dim\tau\ge n-i$.
Thus $N_0$ is a finite set and $N_n=N$.
Similarly let $M_i$ be the union of all $\tau^*$ where $\dim\tau\ge n-i$ and let $f_i=f|_{N_i}\:N_i\to M_i$.
Let $g_0=g|_{N_0}$.
Let us assume that $g_{i-1}\:N_{i-1}\to\R$ is a smooth function such that 
$f_{i-1}\x g_{i-1}\:N_{i-1}\to M_{i-1}\x\R$ is a smooth embedding and $\tilde g_{i-1}\:\Delta_{f_{i-1}}\to S^0$ 
coincides with $\tilde g|_{\Delta_{f_{i-1}}}$.

Let us fix an $(n-i)$-simplex $\tau$ of $L$ and an $(n-i)$-simplex $\sigma$ of $K$ such that $f(\sigma)=\tau$.
Let $f_\sigma=f|_{\sigma^*}$ and $g_\sigma=g|_{\sigma^*}$.
By Lemma \ref{morin projection} and Morin's normal form \cite{Mo} $f_\sigma\:\sigma^*\to\tau^*$ is 
$C^\infty$-left-right-equivalent to $F_r\:\R^i\to\R^{m-n+i}$ for some $r\ge 0$.
Let $e_\sigma\:\sigma^*\to I_\sigma$ be a smooth function such that 
$f_\sigma\x e_\sigma\:\sigma^*\to\tau^*\x I_\sigma$
is $C^\infty$-left-right-equivalent to $\Phi_r^\delta\:\R^i\to\R^{m-n+i+1}$, where $\delta$ is chosen so that
the resulting trivialization of $\ker d(f_\sigma)$ coincides with the restriction of $\eps$.
Thus $d(e_\sigma)_x(v)$ is of the same sign as $\eps_x(v)$ for each $x\in\Sigma_{f_\sigma}$ and each 
$v\in\ker d(f_\sigma)_x\but\{0\}$.
Let $\gamma\:\R\to\sigma^*$ be a smooth curve such that $\gamma(0)=x$, $\gamma'(0)=v$ and 
$\big(\gamma(t),\gamma(-t)\big)\in\Delta_{f_\sigma}$ for each $t\in\R$.
Then $\tilde e_\sigma\big(\gamma(t),\gamma(-t)\big)$ is also of the same sign as $d(e_\sigma)_x(v)$.
On the other hand, $\eps_x(v)$ is of the same sign as $\tilde g\big(\gamma(t),\gamma(-t)\big)$.
Since $f_\sigma$ is $C^\infty$-left-right-equivalent to $F_r$ (or alternatively, using a transversality 
theorem of F.~Ronga, see \cite{partII}*{proof of Corollary \ref{II:smooth-delta}}),
$\Delta_{\Sigma_{f_\sigma}}$ is a submanifold of the closure of $\Delta_{f_\sigma}$.
It follows that $\tilde e_\sigma$ and $\tilde g$ coincide on a punctured neighborhood%
\footnote{A {\it punctured neighborhood} of $A$ in $X$ is a set of the form $U\but A$, where $U$ is 
a neighborhood of $A$ in $X$.}
of $\Delta_{\Sigma_{f_\sigma}}$ in the closure of $\Delta_{f_\sigma}$.
In particular, they coincide on all pairs $(x_\eps,y_\eps)\in\Delta_{f_\sigma}$ that are sufficiently close to 
$(\hat\sigma,\hat\sigma)$.
But for any pair $(x,y)\in\Delta_{f_\sigma}$ we have $x=(1-t_0)\hat\sigma+t_0x'$ and $y=(1-t_0)\hat\sigma+t_0y'$ 
for some $x',y'\in\partial\sigma^*$ and some $t_0\in I$.
If $x_t=(1-t)\hat\sigma+tx'$ and $y_t=(1-t)\hat\sigma+ty'$ for some $t\in [\eps,t_0]$, 
then clearly $(x_t,y_t)\in\Delta_{f_\sigma}$.
It follows that $\tilde e_\sigma$ coincides with $\tilde g|_{\Delta_{f_\sigma}}$.

Now $\tilde g$ also coincides with $\tilde g_{i-1}$ on $\Delta_{f|_{\partial\sigma^*}}$.
Therefore by Lemma \ref{homotopy}(b) $e_\sigma|_{\partial\sigma^*}$ is homotopic to $g_{i-1}|_{\partial\sigma^*}$ 
by a homotopy $h_t$ such that $f|_{\partial\sigma^*}\x h_t\:\partial\sigma^*\to\partial\tau^*\x\R$
is a smooth isotopy.
Using $h_t$, it is not hard to construct a smooth function $e'_\sigma\:\sigma^*\to I_\sigma$ which coincides 
with $g_{i-1}$ on $\partial\sigma^*$ and a homotopy $h'_t\:\sigma^*\to I_\sigma$ from $e_\sigma$ to $e'_\sigma$ 
such that $f_\sigma\x h'_t\:\sigma^*\to\tau^*\x\R$ is a smooth isotopy.
The existence of $h'_t$ implies that $\tilde e'_\sigma$ coincides with $\tilde g|_{\Delta_{f_\sigma}}$.
It follows that $g_{i-1}$ extends to a smooth function $g_i\:N_i\to\R$ such that $f_i\x g_i\:N_i\to M_i\x\R$ 
is a smooth embedding and $\tilde g_i\:\Delta_{f_i}\to S^0$ coincides with $\tilde g|_{\Delta_{f_i}}$.

In the end we obtain a smooth function $g_n\:N\to\R$ such that $f\x g_n\:N\to M\x\R$ is a smooth embedding and 
$\tilde g_n\:\Delta_f\to S^0$ coincides with $\tilde g$.
By Lemma \ref{homotopy}(a) $g_n$ is homotopic to $g$ by a homotopy $H_t$ such that $f\x H_t\:N\to M\x\R$ is
an isotopy.
\end{proof}

A simplified version of the proof of Theorem \ref{smoothing}, without references to Lemmas \ref{triangulation} 
and \ref{morin projection}, establishes the following

\begin{theorem}[Gromov, Sz\H ucs] \label{szucs} Let $f\:N^n\to M^m$ be a stable smooth map between smooth manifolds, 
where $n\le m$.
Then $f$ lifts to a smooth immersion $N\to M\x\R$ if and only if $f$ is a corank one map and $\ker df$ 
is trivial as a line bundle over $\Sigma_f$.
\end{theorem}

The case $n=2$, $m=\R^2$ was proved by Haefliger \cite{Hae}, the case $n=m$ by Blank and Curley \cite{BC}*{proof of Theorem 1} 
(see also \cite{Mill}, \cite{Lu}), and the case $M=\R^m$ by Gromov \cite{Gr}*{Lemma 2.1.1(A)}.
The case $M=\R^m$ and the general case were routinely stated without proof (``it is easy to see'') in a number of papers by A. Sz\H ucs 
and his collaborators, starting from 1991 \cite{Sz}*{p.\ 344}.
G. Lippner's dissertation, supervised by Sz\H ucs, claims that ``we will later see'' a proof of the general case \cite{Li1}*{p.\ 5}, 
but I could not find it.
The proof is indeed rather easy, but not entirely trivial.

The referee has kindly provided a self-contained proof Theorem \ref{szucs} not involving any PL topology, which is similar to Gromov's proof, 
but much more explicit (see \S\ref{smoothing-referee} below).

\begin{theorem} \label{imm-smoothing} Let $f\:N^n\to M^m$, $n\le m$, be a corank one stable smooth map from 
a compact smooth manifold to a smooth manifold. 
The following are equivalent:
\begin{enumerate}
\item $f$ lifts to a smooth immersion $N\to M\x\R^k$;

\item $f$ lifts to a topological immersion $N\to M\x\R^k$;

\item the line bundle $\ker df$ over $\Sigma_f$ admits a monomorphism to the trivial bundle 
$\Sigma_f\x\R^k\to\Sigma_f$ lying over $\id\:\Sigma_f\to\Sigma_f$.
\end{enumerate}
\end{theorem}

As one could expect, it is not hard to prove Theorem \ref{imm-smoothing} by elaborating on Gromov's proof of the Smale--Hirsch 
$h$-principle \cite{Gr}*{2.1.1(C)}.
The author learned this from the referee, who kindly provided some details of a proof along these lines 
(see \S\ref{smoothing-referee} below).

The author's original proof of Theorem \ref{imm-smoothing} is in the spirit of the proof of Theorem \ref{smoothing} 
(and so involves PL topology).
It appeared in attacking Problem \ref{prob}, and could be of more interest than Gromov's argument in 
the context of this problem.

\begin{proof} Clearly (1) implies (2).
The implication (2)$\Rightarrow$(3) is proved in Lemma \ref{morin projection} in the case $k=1$, and
the general case is similar.
It remains to prove (3)$\Rightarrow$(1).

By a theorem of A. Verona \cite{Ve} (see also \cite{partII}*{Theorem \ref{II:Shiota}}) 
we may assume that $f$ is simplicial in some smooth triangulations $K$ and $L$ of $N$ and $M$.
Then $\Sigma_f$ is triangulated by a subcomplex of $K$; by passing to barycentric subdivisions if necessary 
we may assume that $\Sigma_f$ is triangulated by a full subcomplex of $K$.
Let $K'$ and $L'$ be the barycentric subdivisions of $K$ and $L$.
If $\sigma$ is a simplex of $K$ and $\tau$ is a simplex of $L$, we write $\sigma^*$, $\tau^*$ for their
dual cones with respect to $K'$ and $L'$.
Like in the proof of Theorem \ref{smoothing}, we may assume that these are smooth disks.
Since $\Sigma_f$ is triangulated by a full subcomplex of $K$, the union $S$ of all $f(v)^*$ such that 
$v\in\Sigma_f$ is a regular neighborhood of $\Sigma_f$ in $N$.
Since $S$ deformation retracts onto $\Sigma_f$, the line bundle $\ker df$ over $\Sigma_f$ uniquely
extends to a line bundle $\Lambda$ over $S$.
Since $f$ immerses $N\but S$, it suffices to construct a lift of $f|_S$ to an immersion $S\to M\x\R^k$.

Let us define a polyhedron $N_f$ as follows.
We start from the disjoint union $\bigsqcup_v f(v)^*$ of the dual cones $f(v)^*$ corresponding 
to each vertex $v$ of $K$.
(Thus if $f(u)=f(v)$ but $u\ne v$, then $f(u)^*$ and $f(v)^*$ are homeomorphic but distinct subsets 
of the disjoint union.)
Corresponding to each simplex $\sigma$ of $K$ we identify the copies of $f(\sigma)^*$ in all the dual cones
$f(v)^*$, where $v$ is a vertex of $\sigma$.
The map $f$ factors in the obvious way into a composition $N\xr{\phi}N_f\xr{\psi}M$, where $\phi$ is 
a homotopy equivalence.
Clearly, $\phi$ restricts to a homotopy equivalence between $S$ and $T:=\phi(S)$.
Hence there is a line bundle $\lambda$ over $T$ such that $\Lambda\simeq \phi^*(\lambda)$; let us fix 
an isomorphism $\eps\:\Lambda\simeq\phi^*(\lambda)$.
Moreover, $\lambda$ admits a monomorphism $\xi$ into the trivial bundle $T\x\R^k\to T$ lying over $\id\:T\to T$.

Next we construct a lift of $\phi|_S$ to an immersion $\chi$ of $S$ into the total space $E(\lambda)$ similarly 
to the proof of Theorem \ref{szucs}.
In more detail, let $S_i$ be union of $\sigma^*$ for all simplexes $\sigma$ of $K$ such that 
$\sigma\subset\Sigma_f$ and $\dim\sigma\ge n-i$.
Thus $S_0=\emptyset$ and $S_n=S$.
Suppose that $\phi|_{S_i}$ lifts to an immersion $\chi_i\:S_i\to E(\lambda)$ such that $d\chi_i$ restricted to 
$\Lambda|_{S_i}$ agrees with $\eps$.
Let $\sigma$ be an $(n-i-1)$-simplex of $K$ contained in $\Sigma_f$.
Like in the proof of Theorem \ref{smoothing}, $f|_{\sigma^*}$ lifts to an embedding 
$e_\sigma\:\sigma^*\to f(\sigma)^*\x\R\subset E(\lambda)$ such that $de_\sigma$ restricted to 
$\Lambda|_{\sigma^*}$ agrees with $\eps$.
Moreover, the linear homotopy between the restrictions of $e_\sigma$ and $\chi_i$ to $\partial\sigma^*$ is 
a regular homotopy since the differentials of both restricted to $\Lambda|_{\partial\sigma^*}$ agree with
$\eps$ and hence with each other.
This yields an extension of $\chi_i$ to an immersion $\chi_{i+1}\:S_{i+1}\to E(\lambda)$ lifting $\phi|_{S_{i+1}}$
and such that $d\chi_{i+1}$ restricted to $\Lambda|_{S_{i+1}}$ agrees with $\eps$.
In the end we obtain an immersion $\chi\:S\to E(\lambda)$ lifting $\phi|_S$, and it is clear from its 
construction that the composition $S\xr{\chi}E(\lambda)\xr{\xi}T\x\R^k\xr{\psi\x\id_{\R^k}}M\x\R^k$ is 
a smooth immersion which lifts $f|_S$.
\end{proof}

\section{Referee's alternative proofs using smooth topology} \label{smoothing-referee}

This section contains proofs of Theorems \ref{szucs} and \ref{imm-smoothing} kindly provided by the referee
(with a few minor corrections and clarifications by the author of the paper).

\begin{proof}[Proof of Theorem \ref{szucs}]
If the given map $f\:N\to M$ lifts to a smooth immersion $N\to M\x\R$, then clearly $f$ is a corank one map and $\ker df$ is trivial 
as a line bundle over $\Sigma_f$ (see also Lemma \ref{morin projection}). 
It remains to prove the converse.
Thus we suppose that $f$ is a corank one map and $\ker df$ is trivial.
Let us note that lifts of $f$ to smooth immersions $N\to M\x\R$ correspond to smooth functions $h\:N\to\R$ whose derivative 
in the direction of $\ker df$ is nowhere zero.

\begin{lemma} \label{referees-lemma}
Let $\Sigma$ be a closed submanifold of the closed manifold $N$ and let $v$ be a smooth vector field along $\Sigma$ 
(that is, $v$ is a section of the bundle $TN|_{\Sigma}$); its value at $p\in\Sigma$ will be denoted by $v_p$. 

(a) Suppose that $v$ is nowhere tangent to $\Sigma$. 
Then there exists a smooth function $h\:N\to\R$ such that $dh(v_p)>0$ for all $p\in\Sigma$.

(b) Under the same hypothesis the values of $h$ on $\Sigma$ can be prescribed in advance; that is, for any smooth function 
$\phi\:\Sigma\to\R$ the function $h$ can be chosen to satisfy $h|_\Sigma=\phi$.

(c) Suppose that $v$ is nowhere vanishing.
If a smooth function $g\:N\to\R$ satisfies $dg(v_p)>0$ on a compact set $C\subset N$ that contains all the points of tangency of $v$ to $\Sigma$, 
then $g|_C$ extends to a smooth function $h\:N\to\R$ such that $dh(v_p)>0$ for all $p\in\Sigma$.
\end{lemma}

\begin{proof}[Proof. (a)] Fix a complete Riemannian metric on $N$, and denote by $\gamma_{v_p}(t)$ (where $p\in N$, $t\in\R$) 
the geodesic through $\gamma_{v_p}(0)=p$ that has derivative $v_p$ at $t=0$. 
For a sufficiently small number $\eps>0$ the smooth map $\psi\:\Sigma\x(-\eps,\eps)\to N$ defined by \[\psi(p,t)=\gamma_{v_p}(t)\] 
is a smooth embedding. 
The function $\tilde h\:\psi(p,t)\mapsto t$ is a smooth function on the image of $\psi$ such that $d\tilde h(v)>0$ holds everywhere. 
Finally, $\tilde h$ restricted to $\Sigma\x[-\frac\eps2,\frac\eps2]$ extends to a smooth function $h\:N\to\R$.
\end{proof}

\begin{proof}[(b)] The same proof as in part (a) works with the modified definition $\tilde h\:\psi(p,t)\mapsto t+\phi(p)$.
\end{proof}

\begin{proof}[(c)] Extend $v$ to a nowhere vanishing vector field in a neighbourhood of $\Sigma$. 
Since $dg(v_p)>0$ is an open condition, it holds on an open neighbourhood $U$ of $C$ in $N$. 
Pick smaller open neighbourhoods $V$, $W$, $X$ and $Y$ of $C$ such that $\Cl Y\subset X\subset\Cl X\subset W\subset\Cl W\subset V\subset\Cl V\subset U$.
We can now repeat the construction in the proof of (a) away from $C$: the function $\psi(p,t)=\gamma_{v_p}(t)$ defined there is an embedding of 
$(\Sigma\but Y)\x(-\eps,\eps)$, which is smooth on $(\Sigma\but\Cl Y)\x(-\eps,\eps)$, and we can define $\tilde h_0\big(\psi(p,t)\big)=t$ 
on its image and extend $\tilde h_0$ restricted to $(\Sigma\but X)\x[-\frac\eps2,\frac\eps2]$ to a smooth function $h_0\:N\to\R$ such that 
$dh_0(v_p)=1$ for all $p\in\Sigma\but X$.

If $\eps$ is greater than the distance between $\Cl X$ and $N\but W$ or the distance between $\Cl V$ and $N\but U$, then decrease $\eps$ 
until this is no longer the case. 
We define a smooth weight function $\phi\:N\to[0,1]$ that satisfies
\begin{itemize}
\item $\phi|_{N\but U}=1$,
\item $\phi|_{\Cl X}=0$ and
\item $d\phi(v_p)=0$ for all $p\in\Sigma$.
\end{itemize}

To do so, we first define $\phi$ on $\Sigma$ to satisfy $\phi|_{\Sigma\but V}=1$ and $\phi|_{\Sigma\cap\Cl W}=0$, 
then extend it to be constant on the curves $\gamma_{v_p}(-\eps,\eps)$ for all $p\in\Sigma$. 
Since no such curve starting from a point $p\in\Sigma\but W$ intersects $\Cl X$ and no curve starting from $p\in\Sigma\cap\Cl V$ 
intersects $N\but U$, we can finally extend $\phi$ to the rest of $N$ in a way that respects the first two conditions.

Define now $h=\beta\phi h_0+g$ for a large constant $\beta\in\R$ to be chosen later.
The directional derivative that we want to make positive for all $p\in\Sigma$ is
\[dh(v_p)=\beta d(\phi h_{0})(v_p)+dg(v_p)=\beta\phi dh_{0}(v_p)+dg(v_p).\]
For $p\in\Sigma\cap X$ this expression is $dg(v_p)>0$. 
For $p\in(\Sigma\cap U)\but X$ the second term is positive and the first term is nonnegative. 
Finally, for $p$ in the remaining compact set $\Sigma\but U$ the values $dg(v_p)$ may be negative, but for some $\beta>0$ we have $ dg(v_p)>-\beta$ 
for all $p\in\Sigma\but U$, and in this case $dh(v_p)=\beta+dg(v_p)>0$ as well. 
\end{proof}

Coming back to the proof of the theorem, let $\Sigma^{1_r}\subset\Sigma^{1_{r-1}}\subset\cdots\subset\Sigma^1=\Sigma_f$ be
the singularity strata of $f$ ($\Sigma^{1_r}$ being an abbreviation for $\Sigma^{\everymath{\scriptstyle}\overbrace{\scriptstyle 1,\dots,1}^r}$), 
with $\Sigma^{1_{r+1}}=\emptyset$.
Let $v$ be a vector field along $\Sigma^1$ directed in the positive direction of $\ker df$ according to its given trivialization. 
Since $\Sigma^{1_{r+1}}$ is empty, $v$ is nowhere tangent to $\Sigma^{1_r}$. 
Hence by Lemma \ref{referees-lemma}(a) we can define a smooth function $h_r\:N\to\R$ such that $dh_r(v_p)>0$ for all $p\in\Sigma^{1_r}$.
Note that $v$ is not tangent to $\Sigma^{1_{r-1}}$ at all points of $\Sigma^{1_{r-1}}\but \Sigma^{1_{r}}$. 
Hence by Lemma \ref{referees-lemma}(c) applied to $C=\Sigma^{1_r}$ and $\Sigma=\Sigma^{1_{r-1}}$, the restriction of $h_r$ to $\Sigma^{1_r}$ 
extends to a smooth function $h_{r-1}\:N\to\R$ such that $dh_{r-1}(v_p)>0$ for all $p\in\Sigma^{1_{r-1}}$.
Iterating this procedure, we arrive at a smooth function $h_1\:N\to\R$ such that $dh_1(v_p)>0$ for all $p\in\Sigma^1=\Sigma_f$. 
Then $(f,h_1)\:N\to M\x \R$ is a smooth immersion which lifts $f$. 
\end{proof}

\begin{proof}[Proof of Theorem \ref{imm-smoothing}, (3)$\Rightarrow$(1)] 
First we construct a bundle monomorphism $TN\to T(M\x\R^k)$. 
Using an arbitrarily chosen Euclidean metric on the bundle $TN$ we can decompose $TN|_\Sigma$ into the direct sum $\ker df\oplus(TN/\ker df)$. 
Let $\pi$ denote the corresponding projection of $TN|_\Sigma$ onto $\ker df$. 
Composing $\pi$ with the given monomorphism $\phi\:\ker df\to\Sigma\x\R^k$ we obtain a bundle morphism $\tilde\alpha\:TN|_\Sigma\to\Sigma\x\R^k$.
Consider the bundle $\Hom(TN,\eps^k)$; its sections are bundle morphisms from $TN$ to the trivial bundle 
$\eps^k=N\x\R^k$ lying over the identity map of $N$. 
The morphism $\tilde\alpha$ can be considered as a partial section of $\Hom(TN,\eps^k)$ defined over $\Sigma$. 
Such a partial section can be extended to a global section (since the fiber is contractible), which yields a morphism $\alpha\:TN\to\eps^k$ 
that coincides with $\tilde\alpha$ over $\Sigma$. 
Taking $\alpha\oplus df$ gives us a fibrewise monomorphism $\Phi\:TN\to T(M\x\R^k)$ lying over $(f,0)\:N\to M\x\R^k$.

The compositions $\omega_i$ of $\alpha$ with the projections $\pi_1,\dots,\pi_k\:\eps^k\to\eps$ onto the $k$ factors are differential $1$-forms on $N$,
which yield a decomposition $\Phi=(df,\omega_1,\dots,\omega_k)$.
Our aim is to replace these $1$-forms by exact $1$-forms $dh_i$, $i=1,\dots,k$, so that $(df,dh_1,\dots,dh_k)$ is still a fibrewise monomorphism. 
If we achieve this, then $(f,h_1,\dots,h_k)\:N\to M\x \R^k$ will be a smooth immersion that lifts $f$.

We now replace $\omega_i$ by $dh_i$ one by one for $i=1,\dots,k$ (compare \cite{Gr}*{proof of 2.1.1(C)}).
First we temporarily omit $\omega_1$. 
The remaining morphism $\Phi_1=(df,\omega_2,\dots,\omega_k)\:TN\to T(M\x \R^{k-1})$ may have a $1$-dimensional kernel on a subset $\Sigma$. 
By the Boardman theory, if $\Phi_1$ is generic, then $\Sigma$ will be a smooth manifold and it will be stratified similarly to the singularity set of 
a stable cornak one map by submanifolds $\Sigma^{1_r}\subset\Sigma^{1_{r-1}}\subset\ldots\subset\Sigma^1=\Sigma$ so that $\ker\Phi_1$ is nowhere 
tangent to $\Sigma^{1_r}$ and for any $p\in\Sigma^{1_i}$ the direction of $\ker\Phi_1$ at $p$ is tangent to $\Sigma^{1_i}$ if and only if 
$p\in\Sigma^{1_{i+1}}$. 
Applying the same argument as in the proof of Theorem \ref{szucs} we construct a function $h_1\:N\to\R$ such that $dh_1$ does not vanish on 
$\ker\Phi_1$ anywhere. 
This is equivalent to $(\Phi_1,dh_1)\:TN\to T(M\x \R^k)$ being a fibrewise monomorphism, so we replace $\omega_1$ with $dh_1$.

Next we temporarily omit $\omega_2$ and consider the remaining bundle morphism $\Phi_2= (df,dh_1,\omega_3,\dots,\omega_k)\:TN\to T(M\x\R^k)$. 
Repeating the previous argument yields a function $h_2$ such that $(df,dh_1,dh_2,\omega_3,\dots,\omega_k)$ is a fibrewise monomorphism, 
and we repeat this step for all $i$ until all the $\omega_i$ are replaced, which gives us the desired smooth immersion that lifts $f$. 
\end{proof}

\section{A visualization of embedded lifts of Morin's normal form}

In conclusion, we discuss the geometry of Morin's normal form $F_r\:\R^n\to\R^m$, $n\le m$, of its embedded lifts 
$\Phi_r^\pm\:\R^n\to\R^{m+1}$ (see \S\ref{smoothing-sec}), and of the space of all its embedded lifts.
The main result of this discussion, Proposition \ref{morin-lift}, has already been proved in \S\ref{smoothing-sec} 
in an easier way.
However, the more explicit proof given below might be useful elsewhere, for instance, in attacking 
Problem \ref{prob}.

Let $T_r$ be the Chebyshov polynomial of the first kind.
It is a degree $r$ polynomial, which is even when $r$ is even and odd when $r$ is odd.
As a map $\R\to\R$, it coincides on $[-1,1]$ with the composition
$[-1,1]\xr{\text{Re}^{-1}_+}S^1\xr{n}S^1\xr{\text{Re}}[-1,1]$ and on $[1,\infty)$ with the composition of 
homeomorphisms $[1,\infty)\xr{\cosh^{-1}}[0,\infty)\xr{n}[0,\infty)\xr{\cosh}[1,\infty)$.

The map $T_r\:\R\to\R$ has two obvious lifts $\Gamma_r^\pm\:\R\to\R^2$, defined by 
$\Gamma_r^\pm(x)=\big(T_r(x),\pm x\big)$, which are smooth embeddings.

\begin{lemma} \label{chebyshov-lift}
The space $C^0(T_r)$ of topological embeddings $\R\to\R^2$ that lift $T_r$ consists of two contractible
path components, one containing $\Gamma_r^+$ and another $\Gamma_r^-$.
\end{lemma}

\begin{proof} Let $m^+_1<m^+_2<\dots$ be the maxima of $T_r$ and $m_1^-<m_2^-<\dots$ be its minima.
Then each $T_r(m_i^+)=1$ and each $T_r(m_i^-)=-1$.
Let $g\:\R\to\R$ be a map such that $T_r\x g\:\R\to\R^2$ is an embedding.
It is easy to see that either
\begin{itemize}
\item $g(m_1^+)<g(m_2^+)<g(m_3^+)<\dots$ and $g(m_1^-)<g(m_2^-)<g(m_3^-)<\dots$; or
\item $g(m_1^+)>g(m_2^+)>g(m_3^+)>\dots$ and $g(m_1^-)>g(m_2^-)>g(m_3^-)>\dots$.
\end{itemize}
It follows that $\Delta:=T_r\x g$ is isotopic through lifts of $T_r$ either to $\Gamma_r^+$ (in the first case)
or to $\Gamma_r^-$ (in the second case).
Namely, the linear homotopy $h_t\:\R\to\R^2$, defined by $h_t(x)=(1-t)\Delta(x)+t\Gamma_r^\eps(x)$,
is clearly an isotopy through lifts of $T_r$.
But this $h_t$ continuously depends on $\Delta$.
\end{proof}

In the case $m=n=r$ the Morin normal form $F_r\:\R^n\to\R^m$ specializes to the map $f_r\:\R^r\to\R^r$, defined 
by $f_r(t_1,\dots,t_{r-1},x)=(t_1,\dots,t_{r-1},\,t_1x+\dots+t_{r-1}x^{r-1}+x^{r+1})$.
In general, $f_r$ can be identified with the restriction of $F_r$ to the plane $t_r=\dots=t_{n-1}=0$.

\begin{lemma} \label{swallowtail} Every component of $\Delta_{F_r}$ contains a point of $\Delta_{f_r}$. 
In fact, $\Delta_{F_r}$ is homeomorphic to $\Delta_{f_r}\x\R^{2n-m-r}$ extending the identification between 
$\Delta_{f_r}$ and $\Delta_{f_r}\x 0$.
\end{lemma}

\begin{proof} A pair of points $(t_1,\dots,t_{r-1},x_1)$ and $(t_1',\dots,t_{r-1}',x_2)$ belongs to
$\Delta_{f_r}$ if and only if each $t_i'=t_i$ and $P(x_1)=P(x_2)$, where $P(x)=t_1x+\dots+t_{r-1}x^{r-1}+x^{r+1}$.
A pair of points $(t_1,\dots,t_{n-1},x_1)$ and $(t_1',\dots,t_{n-1}',x_2)$ belongs to $\Delta_{F_r}$ if and only 
if each $t_i'=t_i$, $P(x_1)=P(x_2)$ and each $Q_i(x_1)=Q_i(x_2)$, where $Q_i(x)=t_{ir}x+\dots+t_{ir+r-1}x^r$.

The condition $Q_i(x_1)=Q_i(x_2)$ is equivalent to saying that $x_1$ and $x_2$ are roots of $Q_i(x)-b$
for some $b\in\R$.
That is, $Q_i(x)=(x-x_1)(x-x_2)R(x)+b$ for some degree $r-2$ polynomial $R(x)$.
Since $Q_i(0)=0$, the latter condition is in turn equivalent to $Q_i(x)=(x-x_1)(x-x_2)R(x)-x_1x_2R(0)$.
Upon substituting $c_0+\dots+c_{r-2}x^{r-2}$ for $R(x)$ we obtain
$Q_i(x)=c_0x(x-x_1-x_2)+(c_1x+\dots+c_{r-2}x^{r-2})(x-x_1)(x-x_2)$.
Here $c_{r-2}$ is uniquely determined as the coefficient at $x^r$; using this, $c_{r-3}$ is uniquely determined
from the coefficient at $x^{r-1}$; and so on.
Thus all the $c_j$, which in more detail can be denoted $c_{ij}$, are independent real parameters.
Also $t_{(m-n+1)},\dots,t_{n-1}$ are additional independent real parameters (these are the coordinates that are 
not used in $P$ and $Q_1,\dots,Q_{m-n}$).
This yields the desired homeomorphism between $\Delta_{F_r}$ and $\Delta_{f_r}\x\R^{2n-m-r}$.
\end{proof}

Let $M_r$ be the set of all polynomials of the form $a_1x+\dots+a_{r-1}x^{r-1}+x^{r+1}$ (which occur
in the definition of the map $f_r$).
If we write $T_{r+1}=c_0+c_1x+\dots+c_{r+1}x^{r+1}$, then $c_{r+1}\ne 0$ and $c_r=c_{r-2}=c_{r-4}=\dots=0$.
Hence $\tau_r:=\frac1{c_{r+1}}(T_{r+1}-c_0)$ belongs to $M_r$.
Let $\lambda\:\R\to\R^r$ be the affine embedding defined by 
$\lambda(x)=(\frac{c_1}{c_{r+1}},\dots,\frac{c_{r-1}}{c_{r+1}},x)$.
Then the restriction of $f_r$ to $\lambda(\R)$ can be identified with $\tau_r$.

The lifts $\gamma_r^\pm\:\R\to\R^2$ of $\tau_r$, defined by $\gamma_r^\pm(x)=\big(\tau_r(x),\pm x\big)$, 
can be identified with the restrictions of $\Phi_r^\pm$.

\begin{lemma} \label{chebyshov} Every component of $\Delta_{f_r}$ contains a point of $\Delta_{\tau_r}$.
\end{lemma}

\begin{proof} 
Let us note that $M_r$ consists of all monic degree $r+1$ polynomials $P$ such that $0$ is a root of $P$ and 
the arithmetic average of all roots of $P$ (including the complex ones) equals $0$.
We may topologize $M_r$ as a set of maps $\R\to\R$ with the $C^\infty$ topology.

In the case $r=1$ we have $\tau_1=f_1$ and there is nothing to prove.
In the case $r=2$ each $P\in M_2$ is of the form $P(x)=x^3-ax$, and if $\Delta_P$ is nonempty, then $a>0$.
We have $T_3=4x^3-3x$, and so $\tau_2=x^3-\frac34x$.
In this case the assertion is obvious, so we will assume that $r\ge 3$.

Let $P\in M_r$ and suppose that $(x_1,x_2)\in\Delta_P$, that is, $x_1\ne x_2$ and $P(x_1)=P(x_2)$.
Since $r\ge 3$, by the proof of Lemma \ref{swallowtail} $P(x)=cx(x-x_1-x_2)+P^+(x)$, where
$P^+(x)=(x-x_1)(x-x_2)Q(x)\in M_r$ and $c\in\R$.
Let $P_t(x)=tcx(x-x_1-x_2)+P^+(x)$.
Then $P_t$, $t\in[0,1]$, is a path in $M_r$ from $P_1=P$ to $P_0=P^+$ such that $(x_1,x_2)\in\Delta_{P_t}$ 
for each $t$ and $P^+$ has $x_1$ and $x_2$ among its roots. 

Next if $P^+$ has less than $r+1$ real roots, then $P^+(x)=(x-a-ib)(x-a+ib)R(x)$ for some $a,b\in\R$.
Let $Q_t(x)=(x-a-itb)(x-a+itb)R(x)$.
Then $Q_t$, $t\in [0,1]$, is a path in $M_r$ from $Q_1=P^+$ to $Q_0=(x-a)^2R(x)$, which has more real roots 
than $Q$, and we have $(x_1,x_2)\in\Delta_{P_t}$ for each $t$.
This procedure can be repeated until we get a path in $M_r$ from $P^+$ to a polynomial 
$P^{++}=(x-x_1)\cdots(x-x_{r+1})$ with all $x_i\in\R$, where $x_1$, $x_2$ are as above.
We may assume that $x_1<x_2$.
Let us note that $x_z=0$ for some $z$ (possibly $z=1$ or $2$).

The roots of $\tau_r$ are also all real.
(Specifically, $\tau_r$ has $r+1$ simple roots $\cos\frac{\pi k/2}{r+1}$, $k=1,\dots,r+1$, if $r$ is even, 
and $\frac{r+1}2$ double roots $\cos\frac{2\pi k}{r+1}$, $k=1,\dots,\frac{r+1}2$, if $r$ is odd.)
We may write $\tau_r=(x-a_1)\cdots(x-a_{r+1})$, where $a_z=0$ and (using that $r\ge 2$) $a_1<a_2$.
Let $R_t=\big(x-(1-t)x_1-ta_1\big)\cdots\big(x-(1-t)x_{r+1}-ta_{r+1}\big)$.
Clearly, $R_t$, $t\in[0,1]$, is a path in $M_r$ from $P^{++}$ to $\tau_r$ and 
$\big((1-t)x_1+ta_1,(1-t)x_2+ta_2\big)\in\Delta_{R_t}$ for each $t$.
\end{proof}

\begin{proposition} \label{morin-lift}
The space $C^0(F_r)$ of topological embeddings $\R^n\to\R^{m+1}$ that lift $F_r$ consists of two contractible
path components, one containing $\Phi_r^+$ and another $\Phi_r^-$.
\end{proposition}

\begin{proof}
Given a point $(x_1,x_2)\in\Delta_{F_r}$, by Lemmas \ref{swallowtail} and \ref{chebyshov} it lies in 
the same component of $\Delta_{F_r}$ with some $(x_1',x_2')\in\Delta_{\tau_r}$.
Given a embedding $\Psi\:\R^n\to\R^{m+1}$ that lifts $F_r$, by Lemma \ref{chebyshov-lift} the restriction
$\delta\:\R\to\R^2$ of $\Psi$ over $\lambda(\R)\x 0$ is isotopic through lifts of $\tau_r$ to $\gamma_r^\eps$
for some sign $\eps$.
By symmetry we may assume that $\Psi(x_1)$ lies above $\Psi(x_2)$.
Then $\delta(x_1')$ lies above $\delta(x_2')$, and consequently $\gamma_r^\eps(x_1')$ lies above 
$\gamma_r^\eps(x_2')$.
But then also $\Phi_r^\eps(x_1)$ lies above $\Phi_r^\eps(x_2)$.
Hence the linear homotopy $h_t\:\R^n\to\R^{m+1}$ between $\Psi$ and $\Phi_r^\eps$, defined by 
$h_t(x)=(1-t)\Psi(x)+t\Phi_r^\eps(x)$, is an isotopy.
But this isotopy through lifts of $F_r$ continuously depends on $\Psi$.
\end{proof}

\section*{Acknowledgements}

I am grateful to P. Akhmetiev, R. Sadykov and A. Sz\H ucs for stimulating 
conversations and useful remarks and to the referee for a very thoughtful review including several pages of alternative arguments.

\section*{Disclaimer} 

I oppose all wars, including those wars that are initiated by governments
at the time when they directly or indirectly support my research. Wars of the latter type
include all wars waged by the Russian state since the Second Chechen war till the
present day (May 7, 2022), as well as the US-led invasions of Afghanistan and Iraq.

\end{document}